\documentclass[10pt]{amsart}

\usepackage{amscd}
\usepackage{amssymb}
\usepackage{latexsym}
\usepackage{graphicx}
\usepackage[parfill]{parskip} 
\usepackage{amsmath}
\usepackage{amsfonts}
\usepackage{amssymb}
\usepackage{latexsym}
\usepackage{graphicx}
\usepackage{enumerate}
\usepackage{bbm}
\usepackage{fourier}

\usepackage[utf8]{inputenc}

\usepackage[final]{hyperref}
\hypersetup{colorlinks=true, linkcolor=blue, anchorcolor=blue, citecolor=red, 
filecolor=blue, menucolor=blue, pagecolor=blue, urlcolor=blue}

\newcommand{\norm}[1]{\left\Vert #1 \right\Vert}

\newcommand{\ls}{\lesssim}

\newcommand{\R}{\mathbb{R}}

\newcommand{\angles}[1]{\langle #1 \rangle}

\newtheorem*{PWtheorem}{Paley-Wiener Theorem}

\newtheorem{theorem}{Theorem}

\newtheorem{lemma}{Lemma}
\newtheorem{corollary}{Corollary}

\theoremstyle{definition}

\theoremstyle{remark}
\newtheorem{remark}{Remark}

\numberwithin{equation}{section}
\begin{document}

\title[gKdV-3: spatial analyticity]{ On the radius of spatial analyticity for the quartic generalized KdV equation}

\author{Sigmund Selberg}
\email{sigmund.selberg@uib.no}
\author{Achenef Tesfahun}
\email{Achenef.Temesgen@uib.no}
 
\address{Department of Mathematics\\
University of Bergen\\
PO Box 7803\\
5020 Bergen\\ Norway}




\date{July 26, 2016}


\begin{abstract}
Lower bound on the rate of decrease in time of the  uniform radius of spatial analyticity of solutions to the quartic generalized KdV equation is derived,
which improves an earlier result by Bona, Gruji\'c and Kalisch. 
\end{abstract}

\maketitle

\section{Introduction}
Consider the Cauchy problem for the quartic generalized Korteweg-de Vries (KdV) equation
\begin{equation}\label{gkdv}
\begin{cases}
u_t + u_{xxx} + (u^4)_x  = 0, \qquad t,x \in \mathbb R,\\
u(x,0) = u_0(x),
\end{cases}
\end{equation}
where the unknown $u(x,t)$ and the datum $u_0(x)$ are real-valued.

 In \cite{g05}  Gr\"{u}nrock proved that the Cauchy problem \eqref{gkdv} is locally well-posed  for data  $u_0\in H^s(\R)$ with $s>-1/6$ and globally well-posed for data
  $u_0\in H^s(\R)$ with $s\ge 0$. Later, Tao \cite{t07} proved that \eqref{gkdv} is globally well-posed for data in the critical space $ \dot H^{-\frac16}(\R)$ with small norm. For an earlier study of well-posedness for \eqref{gkdv} we refer to \cite{kpv93}.
  
 In the present paper we shall study spatial analyticity of the solutions to the above Cauchy problem motivated by earlier works on this issue 
 for generalized KDV by Bona, Gruji\'c and Kalisch \cite{bgk05} and a recent one for KDV by Selberg and Da Silva \cite{sd16}.  In particular, we consider a real-analytic initial data $u_0$ with 
 uniform radius of analyticity $\sigma_0>0$, so there is a holomorphic extension to a complex strip 
 $$ S_{\sigma_0} =\{x + iy : |y| < \sigma_0 \}.$$ The question is then whether this property persists for all later times $t$, but with a possibly smaller and shrinking radius of analyticity $\sigma(t) > 0$, i.e. is the solution $u(t,x)$ of \eqref{gkdv} analytic in $S_{\sigma(t)}$ for all $t$? 
   For short times it was shown by Gruji\'c and Kalisch in \cite{gk02} that the radius of analyticity remains at least as large as the initial radius, i.e. one can take $\sigma(t)=\sigma_0$. 
For large times on the other hand it was shown by Bona, Gruji\'c and Kalisch in \cite[see Corollary 4]{bgk05} that $\sigma(t)$ can decay no faster than \footnote{ We use the notation $a\pm=a\pm \varepsilon$ for sufficiently small $\varepsilon>0$. } 
$t^{-164-}$ as $t \to \infty$.  In this paper we use the idea introduced in \cite{st15} (see also \cite{sd16}) to improve this result significantly showing $\sigma(t)$ can decay no faster than 
$t^{-2}$ as $t \to \infty$.
 For studies on related issues for nonlinear partial differential equations see for instance \cite{am84, cdn14,  hhp11,  hp12,  j86, km86, p12, l12}.

 The Gevrey space, denoted $G^{\sigma,s} = G^{\sigma,s}(\mathbb R)$, is a suitable space to study analyticity of solution. This space is defined by the norm
\[
\| f \|_{G^{\sigma, s}(\mathbb R)} = \left\| e^{\sigma | \xi |}\angles{ \xi }^s \hat{f}(\xi) \right\|_{L^2_\xi (\mathbb{R})},
\]
where $\hat f$ denotes the Fourier transform given by
$$
\hat f( \xi)=\int_{\R} e^{ -i x\xi} f(x) dx
$$
and $\angles{\xi}=\sqrt{1+|\xi|^2}$. For $\sigma=0$ the Gevrey-space coincides with the Sobolev space $H^s$.  We shall  write $G^{\sigma}=G^{\sigma,0}$.
One of the key properties of the Gevrey space is that every function in $G^{\sigma,s}$ with $\sigma>0$ has an analytic extension to the strip $S_\sigma$. This property 
is contained in the following.
\begin{PWtheorem}
Let $\sigma > 0$, $s \in \mathbb R$.  Then the following are equivalent:
\begin{enumerate}
\item $f \in G^{\sigma, s}$.
\item $f$ is the restriction to the real line of a function $F$ which is holomorphic in the strip
\[
S_\sigma =  \{ x + i y :\ x,y \in \mathbb{R}, | y | < \sigma\}
\]
and satisfies
\[
\sup_{| y | < \sigma} \| F( x + i y ) \|_{H^s_x} < \infty.
\]
\end{enumerate}
\end{PWtheorem}
The proof given for $s = 0$ in \cite[p. 209]{k76} applies also for $s\in \R$ with some obvious
modifications.

Observe that the Gevrey spaces satisfy the following embedding property:
\begin{align}
\label{Gembedding}
  G^{\sigma,s} &\subset G^{\sigma',s'} \quad \text{for all $0 \le \sigma' < \sigma$ and $s,s' \in \mathbb R$}.
\end{align}
In particular, setting $\sigma'=0$, we have the embedding $ G^{\sigma,s} \subset H^{s'}$ for all $0  < \sigma$ and $s, s'\in \mathbb R $.
As a consequence of this property and the existing well-posedness theory in $H^s$ we conclude that the Cauchy problem \eqref{gkdv} has a unique, 
smooth solution for all time, given initial data 
$u_0 \in G^{\sigma_0,s}$ for all $\sigma_0>0$ and $s\in \R$.
Our main result gives an algebraic lower bound on the radius of analyticity $\sigma(t)$ of the solution as the time $t$ tends to infinity.
\begin{theorem}\label{thm-gwp}
Assume $u_0 \in G^{\sigma_0,s}$ for some $\sigma_0 > 0$ and $s\in \R$.  Let $u$ be the global $C^\infty$ solution of \eqref{gkdv}. Then $u$ satisfies
\[
  u(t) \in G^{\sigma(t), s} \quad \text{for all } \ t\in \R,
\]
with the radius of analyticity $\sigma(t)$ satisfying an asymptotic lower bound
\[
\sigma(t) \ge ct^{-2}\quad \text {as} \ |t|\rightarrow \infty,
\]
where $c > 0$ is a constant depending on  $\|u_0\|_{G^{\sigma_0,s}}$, $\sigma_0$ and $s$.
\end{theorem}

We note that \eqref{gkdv} is invariant under the reflection $(t,x) \to (-t,-x)$. Hence we may from now on restrict ourselves to positive times $t\ge 0$. 
The first step in the proof of Theorem \ref{thm-gwp} is to show that in a short time interval $0 \le  t \le  \delta$, where $\delta > 0$ depends on the norm of the initial data, 
the radius of analyticity remains constant. This is proved by a contraction argument involving energy estimates, 
Sobolev embedding and a multilinear estimate that is similar to the one proved by Gr\"{u}nrock in \cite{g05}. This result is stated in Section 2.
The next step is to improve the control on the growth of the solution in the time interval $[0, \delta]$, measured in the data norm $G^{\sigma_0}$.
To achieve this we show that, although the conservation of  $G^{\sigma_0}$-norm of solution does not hold exactly, it does hold in an approximate sense (see Section 3).
This approximate conservation law will allow us to iterate the local result and obtain Theorem \ref{thm-gwp}. 
This will be proved in the last Section 4.

\section{Preliminaries}

\subsection{ Function spaces}

Define the Bourgain space $X^{s,b}$ by the norm
\begin{align*}
\| u \|_{X^{s,b}} &= \left\| \angles{ \xi }^{s} \angles{ \tau - \xi^3}^b \widetilde{u} (\xi, \tau) \right\|_{L^2_{\tau, \xi}},
\end{align*}
where $\widetilde u$ denotes the space-time Fourier transform given by
$$
\widetilde u(\tau, \xi)=\int_{\R^2} e^{ i(t\tau+ x\xi)} u(t,x) dt dx.
$$
The restriction to time slab $\mathbb R \times (0, \delta)$ of the Bourgain space, denoted $X^{s, b}_\delta$, is a Banach space when equipped with the norm
$$
\| u \|_{X^{s, b}_\delta}  = \inf \left\{ \| v \|_{X^{s, b}}: \ v = u \text{ on } \mathbb{R} \times (0, \delta) \right\}.
$$

In addition, we also need the Grevey-Bourgain space,  denoted $X^{\sigma, s, b}$, defined by the norm
\begin{align*}
\| u \|_{X^{\sigma, s, b}} &= \left\| e^{\sigma | D_x |}u \right\|_{X^{s, b}} ,
\end{align*}
where $D_x = -i\partial_x$, which has Fourier symbol $\xi$. In the case $\sigma=0$, this space coincides with the Bourgain space $ X^{s,b}$. 
The restrictions of $X^{\sigma, s, b}$ to a time slab $\mathbb R \times (0, \delta)$, denoted $X^{\sigma, s, b}_\delta$, is defined in a similar way as above.

\subsection{Linear estimates}
In this subsection we collect linear estimates needed to prove local existence of solution. The $X^{\sigma, s, b}$- estimates given below easily follows 
by substitution $u\rightarrow e^{\sigma | D_x |}u$ from the properties of $X^{s, b}$-spaces (and its restrictions).  In the case $\sigma=0$, the
proofs of the first two lemmas below can be found in section 2.6 of \cite{t06}, 
 whereas the third lemma follows by the argument used to prove Lemma 3.1 of \cite{ckstt04} and the fourth lemma is the standard energy estimate in 
 $X^{s, b}_\delta$-spaces.

\begin{lemma}\label{embed}
Let $\sigma \ge 0$, $s \in \mathbb{R}$ and $b > 1/2$. Then $X^{\sigma,s,b} \subset C(\mathbb R,G^{\sigma,s})$ and
\[
  \sup_{t \in \mathbb{R}} \| u(t) \|_{G^{\sigma, s}} \leq C \| u \|_{X^{\sigma, s, b}},
\]
where the constant $C > 0$ depends only on $b$.
\end{lemma}

\begin{lemma}\label{exponent}
Let $\sigma \ge 0$, $s \in \mathbb{R}$, $-1/2 < b < b' < 1/2$ and $\delta > 0$.  Then
\[
\| u \|_{X^{\sigma, s, b}_\delta} \leq C \delta^{b' - b} \| u \|_{X^{\sigma, s, b'}_\delta},
\]
where $C$ depends only on $b$ and $b'$.
\end{lemma}

\begin{lemma}\label{restrict}
Let $\sigma \ge 0$, $s \in \mathbb R$, $-1/2 < b < 1/2$ and $\delta > 0$. Then for any time interval $I \subset [0,\delta]$ we have
\[
  \norm{\chi_{I} u}_{X^{\sigma,s,b}} \le C \norm{u}_{X^{\sigma,s,b}_\delta},
\]
where $\chi_I(t)$ is the characteristic function of $I$, and $C$ depends only on $b$.
\end{lemma}

Next, consider the linear Cauchy problem, for given $g(x,t)$ and $u_0(x)$,
\[\begin{cases}
u_t + u_{xxx} = g, \\
u(0) = u_0.
\end{cases}\]
Let $W(t) = e^{- t\partial_{x}^3} = e^{itD_x^3}$ be the solution group with Fourier symbol $e^{it\xi^3}$.  Then
we can write the solution using the Duhamel formula
\[
u(t) = W(t)u_0 + \int_0^t W(t - t') g(t') \ dt'.
\]
 Then $u$ satisfies the following $X^{\sigma, s, b}$ energy estimate.
\begin{lemma}\label{linear}
Let $\sigma \ge 0$, $s \in \mathbb{R}$, $1/2 < b \leq 1$ and $0 < \delta \leq 1$.  Then for all $u_0 \in G^{\sigma, s}$ and $F \in X^{\sigma, s, b-1}_\delta$, we have the estimates
\[\begin{aligned}
& \qquad \qquad \| W(t) u_0 \|_{X^{\sigma, s, b}_\delta} \leq C \| u_0 \|_{G^{\sigma, s}},
\\
& \left\| \int_0^t W(t - t') g(t')\ dt' \right\|_{X^{\sigma, s, b}_\delta} \leq C \| g \|_{X^{\sigma, s, b - 1}_\delta},
\end{aligned}\]
where the constant $C > 0$ depends only on $b$.
\end{lemma}

\subsection{Multilinear estimates and local result}
The following multilinear estimates due to Gr\"{u}nrock \cite{g05} and  Gr\"{u}nrock, Panthee and Silva \cite{gps07} are key for proving our main result.
\begin{lemma} \cite[Theorem 1 and Corollary 2 ]{g05} \label{lm-nonlinearest-a}
Let $s>-\frac16$. Assume $-\frac12<b'< s-\frac13$ if $-\frac16<s\le 0$ and $-\frac12<b'< -\frac13$ if $s\ge 0$.
Then for all $b>\frac12$ we have
\begin{equation}\label{nonlinearesta}
\|   \partial_x\left(\prod_{j=1}^4 u_j \right) \|_{X^{s, b'}}\ls  \prod_{j=1}^4 \|   u_j  \|_{X^{s, b}}.
\end{equation}
\end{lemma}

\begin{lemma}\cite[Lemma 2]{gps07} \label{lm-nonlinearest-b}
Let $b>\frac12$,  $s_j\le 0$ for $j=1, \cdots, 4$ and $\sum_{j=1}^4 s_j= -\frac12$.  Then
\begin{equation}\label{nonlinearestb}
\|   \partial_x\left(\prod_{j=1}^4 u_j \right) \|_{X^{0, -b}}\ls  \prod_{j=1}^4 \|   u_j  \|_{X^{s_j, b}}.
\end{equation}
\end{lemma}

From Lemma \ref{lm-nonlinearest-a} and a simple triangle inequality we obtain the following.
\begin{corollary}\label{col-nonlinearest}
Let $s$, $b$ and $b'$ be as in Lemma \ref{lm-nonlinearest-a}. Then for all $\sigma\ge 0$ we have
\begin{equation}\label{nonlinearest1}
\|   \partial_x\left(\prod_{j=1}^4 u_j \right) \|_{X^{\sigma, s, b'}}\le C  \prod_{j=1}^4 \|   u_j  \|_{X^{\sigma, s, b}},
\end{equation}
where $C$ is independent of $\sigma$.
\end{corollary}
\begin{proof}
Let 
$$
\widehat{v_j}(\tau, \xi):= e^{\sigma|\xi|}   \widehat{u_j}(\tau, \xi),
$$
then \eqref{nonlinearest1} is reduced to
$$
\|I\|_{L^2_{\tau, \xi}}\ls  \prod_{j=1}^4 \|   v_j  \|_{X^{ s, b}}
$$
where 
$$
I (\tau, \xi)=i\xi  \angles{\xi}^s  \angles{\tau-\xi^3}^{b'} \int_{\ast}   e^{\sigma ( |\xi| -\sum_{j=1}^4 |\xi_j| )} \prod_{j=1}^4 \widehat{v_j}(\tau_j, \xi_j)d\tau_\ast d\xi_\ast,
$$
where we used the notation
\begin{equation}\label{int-notation}
\int_{\ast}  w d\tau_\ast d\xi_\ast=
 \int_{\sum_{j=1}^{4} \xi_j =\xi, \ \sum_{j=1}^{4} \tau_j =\tau} w \prod_{j=1}^3  d\tau_j d\xi_j 
\end{equation}
for a function $w=w(\tau_j, \xi_j)$.
 By the triangle inequality we have
$|\xi|\le \sum_{j=1}^4 |\xi_j|$ which implies  $e^{\sigma ( |\xi| -\sum_{j=1}^4 |\xi_j| )}\le 1$, and hence
$$
\|I\|_{L^2_{\tau, \xi}} \ls  \|   \partial_x\left(\prod_{j=1}^4 v_j \right) \|_{X^{s, b'}}.
$$
Thus \eqref{nonlinearest1} is reduced to showing
$$
\|   \partial_x\left(\prod_{j=1}^4 v_j \right) \|_{X^{s, b'}}\ls  \prod_{j=1}^4 \|   v_j  \|_{X^{ s, b}}
$$
which is \eqref{nonlinearesta}.

\end{proof}

Then by Picard iteration and Corollary \ref{col-nonlinearest} one obtains the following local result (for details see \cite[proof of Theorem 1 therein]{{sd16}}).
\begin{theorem}\label{thm-lwp}
Let $\sigma > 0$ and $s > -\frac16$. Then for any $u_0 \in G^{\sigma,s}$ there exists a time $\delta = \delta(\| u_0 \|_{G^{\sigma,s}}) > 0$ and a unique solution $u$ of \eqref{gkdv} on the time interval $(0,\delta)$ such that
\[
u \in C([0,\delta], G^{\sigma,s}).
\]
Moreover, the solution depends continuously on the data $u_0$, and we have
\[
  \delta = c_0(1+\|u_0\|_{G^{\sigma,s}})^{-r}
\]
for some constants $c_0 > 0$ and $r > 1$ depending only on $s$. Furthermore, the solution $u$ satisfies the bound 
\begin{equation} \label{solnbound}
\| u\|_{X^{ \sigma, s, b}_\delta}\le C \| u_0\|_{G^{\sigma, s}} \quad \text{for} \ b>\frac12,
\end{equation}  
where $C$ depends only on $s$ and $b$.
\end{theorem}
\begin{remark}\label{rmk-lwp}
Theorem \ref{thm-lwp} shows that if the initial data $u_0$ is analytic on the strip $S_\sigma$ so is the solution $u(t)$
 on the same strip as long as $t\in[0, \delta]$. Note also that in view of the embedding \eqref{Gembedding} we can allow $s\le -\frac16 $ in Theorem \ref{thm-lwp} but then 
 the solution will be analytic only on a slightly smaller strip $S_{\sigma-}$.
 \end{remark}

\section{Almost conservation law }
For a given $u(0)\in G^{\sigma} $ we have by Theorem \ref{thm-lwp} a solution $u(t)\in  G^{\sigma}$ for $0\le t\le \delta$ 
satisfying the bound
\begin{equation}\label{C}
\sup_{t\in [0,  \delta]} \| u(t) \|_{G^{\sigma}}\le C\| u(0) \|_{G^{\sigma}},
\end{equation}
where we also used \eqref{solnbound} and Lemma \ref{embed}; the constant $C$ in \eqref{C} comes from these estimates and is independent of $\delta$ and $\sigma$. 
The question is then whether we can improve on estimate \eqref{C}. In what follows
we will use equation \eqref{gkdv} and Theorem \ref{thm-lwp} to obtain the approximate conservation law
\begin{equation*}
 \sup_{t\in [0,  \delta]} \| u(t) \|^2_{G^{\sigma}} =\| u(0) \|^2_{G^{\sigma}} + E_\sigma(0) ,
\end{equation*} 
 where $E_\sigma(0)$ satisfies the bound 
 $ E_\sigma(0)\le C \sigma^{\frac12}  \| u(0) \|^5_{G^{\sigma}}$. The quantity $E_\sigma(0)$ can be considered an error term since in the limit as $\sigma\rightarrow 0 $, we have 
 $E_\sigma(0) \rightarrow 0$, and hence recovering the well-known conservation of $L^2$-norm of solution: $ \| u(t) \|_{L^2}=  \| u(0) \|_{L^2}$ for all $t\in  [0,  \delta]$.

\begin{theorem}\label{thm-approx}
Let $b>\frac12$ and $\delta$ be as in Theorem \ref{thm-lwp}. Then there exists $C > 0$ such that for any $\sigma > 0$ and any solution
 $u \in X^{\sigma, 0, b}_\delta$ to the Cauchy problem \eqref{gkdv} on the time interval  $[0,\delta]$, we have the estimate
\begin{equation}\label{approx1}
\sup_{t\in [0,  \delta]} \| u(t) \|^2_{G^{\sigma}} \leq \| u(0) \|^2_{G^{\sigma}} + C \sigma^{\frac12} \| u \|^5_{X^{\sigma, 0, b}_\delta}.
\end{equation} 
Moreover, we have
\begin{equation}\label{approx2}
 \sup_{t\in [0,  \delta]} \| u(t) \|^2_{G^{\sigma}} \leq \| u(0) \|^2_{G^{\sigma}} + C \sigma^{\frac12} \| u(0) \|^5_{G^{\sigma}},
\end{equation} 

\end{theorem}
\begin{proof}
The estimate \eqref{approx2} follows from \eqref{approx1} and \eqref{solnbound}. Thus, it remains to prove \eqref{approx1}.

Let $v(t,x)=e^{\sigma |D_x| } u (t,x)$ which is real-valued since the multiplier $e^{\sigma |D_x| }$ is even and $u$ is real-valued. Applying $e^{\sigma |D_x| }$ to \eqref{gkdv} we obtain
\begin{equation}\label{gkdv-m}
v_t + v_{xxx} + \partial_x(v^4)  = f,
\end{equation}
where
$$
f=\partial_x \left\{  (  e^{\sigma |D| } u)^4  -e^{\sigma |D| } \left(u^4\right)  \right\}.
$$
Multiplying \eqref{gkdv-m} by $v$ and integrating in space we obtain
$$
\frac12 \frac{d}{dt} \int_\R v ^2 dt  + \int_\R  \partial_x\left( v v_{xx} - \frac12 v_x^2 +\frac45 v^5\right) dx= \int_\R  vf  dx.
$$
We may assume \footnote{
In general, this property holds by approximation using the monotone convergence theorem and the Riemann-Lebesgue Lemma  whenever $u\in X_\delta^{ \sigma, 0, b}$
(see the argument in \cite[pp. 9 ]{sd16}). }  $v, v_x$ and $v_{xx}$ decays to zero as $|x|\rightarrow \infty$. This in turn implies
$$
\frac{d}{dt}\int_\R v^2 dx= 2\int_\R vf  dx.
$$
Now integrating in time over the interval $[0, \delta]$, we obtain
\begin{align*}
\int_\R v^2(\delta,x) dx= \int_\R  v^2(0,x ) dx+2\int_\R \int_\R \chi_{[0, \delta]}(t) vf  dx dt.
\end{align*}
Thus, 
$$
\| u(\delta)\|_{G^\sigma}^2=\| u(0)\|_{G^\sigma}^2 +2\int_\R \int_\R \chi_{[0, \delta]}(t) vf  dx dt.
$$
We now use Plancherel,  H\"{o}lder,  Lemma \ref{restrict} and Lemma \ref{lm-f:est} below to estimate the integral on the right hand side as
\begin{align*}
\left |\int_\R \int_\R \chi_{[0, \delta]}(t) vf  dx dt \right|
&\le\| v\|_{X^{0, b}_\delta}  \| f\|_{X^{0, -b}_\delta}
\\
&\le \| v\|_{X^{0, b}_\delta}  \cdot  C \sigma^\frac12  \| v\|_{X^{ 0, b}_\delta}^4
\\
& =C\sigma^\frac12  \| u\|_{X^{ \sigma, 0, b}_\delta}^5.
\end{align*}

\end{proof}

\begin{lemma}\label{lm-f:est}
Let
$$
f=\partial_x \left\{  (  e^{\sigma |D_x| } u)^4  -e^{\sigma |D_x| } (u^4)  \right\}.
$$
For $b>\frac12$ we have
$$
 \| f\|_{X^{0, -b}_\delta}\ls \sigma^\frac12  \| v\|_{X^{0, b}_\delta}^4 ,
$$
where $v= e^{\sigma |D_x| } u$.
\end{lemma}

The following estimate is needed to prove Lemma \ref{lm-f:est}.
\begin{lemma}\label{lm-symbol:est}
Let
 $\xi_{\text{min}}$, $\xi_{\text{nd}}$, $\xi_{\text{rd}}$ and $\xi_{\text{max}}$  denote the minimum, second largest, third largest and maximum 
 of $\{ |\xi_1|,  |\xi_2|, |\xi_3|, |\xi_4| \}$. Then for $\theta \in [0, 1]$ we have the estimate
\begin{equation}\label{symbolest}
e^{\sigma \sum_{j=1}^{4} |\xi_j| }  -e^{\sigma |\sum_{j=1}^{4} \xi_j | }  \le \left[ 24 \sigma \xi_{\text{rd}}\right]^{\theta} 
e^{\sigma \sum_{j=1}^{4} |\xi_j| }.
\end{equation}
\end{lemma}
\begin{proof}
First note that for $x\ge 0$ we have 
$$
e^x-1\le e^x \quad \text{and} \quad
e^x-1\le xe^x.
$$
Hence
$$
e^x-1\le x^{\theta} e^x
\quad \text{
for}  \ \theta \in [0,1].$$
This in turn implies 
\begin{align*}
\text{LHS} \ \eqref{symbolest} &=\left\{  e^{ \sigma (\sum_{j=1}^{4} |\xi_j|   -|\sum_{j=1}^{4} \xi_j |)}-1\right\} e^{\sigma |\sum_{j=1}^{4} \xi_j | } 
\\
&\le \sigma^\theta  \left( \sum_{j=1}^{4} |\xi_j|   -|\sum_{j=1}^{4} \xi_j |\right)^\theta  e^{ \sigma \sum_{j=1}^{4} |\xi_j| }.
\end{align*}
Then \eqref{symbolest} follows from the following estimate:
\begin{align*}
 \sum_{j=1}^{4} |\xi_j|   -|\sum_{j=1}^{4} \xi_j |&=\frac{  (\sum_{j=1}^{4} |\xi_j| )^2  - |\sum_{j=1}^{4} \xi_j | ^2}{  \sum_{j=1}^{4} |\xi_j|   +|\sum_{j=1}^{4} \xi_j |}
 \\
 &= \frac{  \sum_{ j =1}^4\sum_{   k =1}^4 (|\xi_j|  |\xi_k| -\xi_j \xi_k )}{  \sum_{j=1}^{4} |\xi_j|   +|\sum_{j=1}^{4} \xi_j |}
 \\
  &\le 24 \frac{ \xi_{\text{rd}} \cdot \xi_{\text{max}} }{  \xi_{\text{max}}}=24  \xi_{\text{rd}} .
\end{align*}

\end{proof}

\begin{proof}[Proof of Lemma \ref{lm-f:est}]

Taking the space-time Fourier Transform of $f$ and using the notation in \eqref{int-notation} we have
\begin{equation*}
|\widetilde{f }(\tau, \xi)| \le |\xi| \int_{\ast}
 \left|    e^{\sigma \sum_{j=1}^{4} |\xi_j| }  -e^{\sigma |\sum_{j=1}^{4} \xi_j | } \right| \prod_{j=1}^4 |\widetilde{u }(\tau_j, \xi_j)| 
 d\tau_\ast d\xi_\ast.
\end{equation*}
 Now
 we use \eqref{symbolest} with $\theta=\frac12$ to obtain
\begin{align*}
|\widetilde{f }(\tau, \xi)| \ls   |\xi|  \int_{\ast}   (\sigma \xi_{\text{rd}} )^\frac12
  \prod_{j=1}^4 |\widetilde{v }(\tau_j, \xi_j)|  d\tau_\ast d\xi_\ast.
\end{align*}
Depending on the relative sizes of $|\xi_j|$, $j=1, \cdots, 4$,  the quantity $\xi_{\text{rd}} $ is either $ |\xi_1|,  |\xi_2|, |\xi_3|$ or $|\xi_4|$. So we obtain
\begin{align*}
 \| f\|_{X^{0, -b}}&= \|   \angles{\tau-\xi^3}^{-b} \widetilde{f }(\tau, \xi) \|_{L^2_{\tau, \xi}}
\\
&\le C  \sigma^\frac12  \|  |\xi| \angles{\tau-\xi^3}^{-b}   \int_{\ast} \xi_{\text{rd}}^\frac12
  \prod_{j=1}^4 |\widetilde{v }(\tau_j, \xi_j)|  d\tau_\ast d\xi_\ast \|_{L^2_{\tau, \xi}}
  \\
  &=C\sigma^\frac12 \|   \partial_x\left\{  v ^3\cdot |D_x|^{\frac12 } v  \right\} \|_{X^{0, -b}}
  \\
  &\le C' \sigma^\frac12   \|     v \|_{X^{0, b}}^3   \|    |D_x|^{\frac12} v \|_{X^{-\frac12, b}} 
  \\
  &\le C' \sigma^\frac12 \|     v \|_{X^{0, b}}^4,
\end{align*}
where in the fourth line we used Lemma \ref{lm-nonlinearest-b}.
\end{proof}

\section{Proof of Theorem \ref{thm-gwp}}
We closely follow the argument in \cite{sd16}. 
First we consider the case $s=0$.
The general case, $s\in \R$, will essentially reduce to $s=0$ as shown in the next subsection. 
\subsection{Case $s=0$}
Let $u_0 \in G^{\sigma_0}$ for some $\sigma_0 > 0$.
Then to construct a solution on $[0, T]$ for  arbitrarily large $T$, we will apply
 the approximate conservation law in Theorem \ref{thm-approx} so as to repeat the local result 
on successive short time intervals to reach $T$, by adjusting the strip width parameter $\sigma$ according to the size of $T$.
By employing this strategy we will show that the solution $u$ to \eqref{gkdv} satisfies
\begin{equation}
\label{uT}
  u(t) \in G^{\sigma(t) } \quad \text{for all }  \  t\in [0,  T] ,
 \end{equation}
with
\begin{equation}
\label{siglb}
\sigma(t) \ge c T^{-2} ,
\end{equation}
where $c > 0$ is a constant depending on $\|u_0\|_{G^{\sigma_0}}$, $\sigma_0$ and $s$.

To this end,
 define
\[
  A_\sigma(t) = \| u(t) \|_{G^{\sigma}},
\]
where $\sigma\in (0, \sigma_0]$ is a parameter to be chosen  later.
By Theorem \ref{thm-lwp}, there is 
a solution $u$ to \eqref{gkdv} satisfying
\[
  u \in C([0, \delta];G^{\sigma_0}), 
\]
where
\begin{equation}\label{delta}
  \delta = c_0 (1+A_{\sigma_0}(0))^{-r} \quad \text{for some } \ r>1.
\end{equation}

Now fix $T$ arbitrarily large. We shall apply the above local result and Theorem \ref{thm-approx} repeatedly, with a uniform time step $\delta$ as in \eqref{delta},  and prove 
\begin{equation}\label{keybound}
\sup_{t\in [0, T]}  A^2_\sigma  (t) \le 2A^2_{\sigma_0} (0)
\end{equation}
for $\sigma$ satisfying \eqref{siglb}. 
Hence we have $A_\sigma(t) < \infty$ for $t \in [0,T]$, and this completes the proof of \eqref{uT}--\eqref{siglb}.

It remains to prove \eqref{keybound}, and this is done as follows. Choose $n \in \mathbb N$ so that $T \in [n\delta,(n+1)\delta)$. Using induction we can show for 
any $k \in \{1,\dots,n+1\}$ that
\begin{align}
  \label{induction1}
  \sup_{t \in [0, k\delta]} A^2_\sigma(t) &\le A^2_\sigma(0) + k C\sigma^\frac12 2^{5/2} A^5_{\sigma_0}(0),
  \\
  \label{induction2}
  \sup_{t \in [0,k\delta]} A^2_\sigma(t) &\le 2A^2_{\sigma_0}(0),
\end{align}
provided $\sigma$ satisfies 
\begin{equation}\label{sigma}
  \frac{2T}{\delta} C\sigma^\frac12 2^{5/2} A^3_{\sigma_0}(0) \le 1.
\end{equation}

Indeed, for $k=1$, we have from Theorem \ref{thm-approx} that
\begin{align*}
  \sup_{t \in [0, \delta]} A^2_\sigma(t) &\le A^2_\sigma(0) +  C\sigma^\frac12  A^5_{\sigma}(0)\le A^2_\sigma(0) +  C\sigma^\frac12 A^5_{\sigma_0}(0),
\end{align*}
where we used $A_\sigma(0) \le A_{\sigma_0}(0) $. This in turn implies \eqref{induction2} provided $C\sigma^\frac12  A^3_{\sigma_0}(0)\le 1$ which holds by \eqref{sigma} since 
$T>\delta$.

Now 
assume \eqref{induction1} and \eqref{induction2} hold for some $k \in \{1,\dots,n\}$. Then by Theorem \ref{thm-approx}, \eqref{induction1} and \eqref{induction2} we have
\begin{align*}
  \sup_{t \in [k\delta, (k+1)\delta]} A^2_\sigma(t) &\le A^2_\sigma(k\delta) +  C\sigma^\frac12  A^5_{\sigma}(k\delta)
  \\
   &\le A^2_\sigma(k\delta) +  C\sigma^\frac12   2^{5/2}  A^5_{\sigma_0}(0)
   \\
      &\le   A^2_\sigma(0) + k C\sigma^\frac12 2^{5/2} A^5_{\sigma_0}(0) +  C\sigma^\frac12   2^{5/2}  A^5_{\sigma}(0).
\end{align*}
Combining this with the induction hypothesis
 \eqref{induction1} (for $k$) we obtain
 \begin{align*}
  \sup_{t \in [0, (k+1)\delta]} A^2_\sigma(t) 
      &\le   A^2_\sigma(0) +( k+1) C\sigma^\frac12 2^{5/2} A^5_{\sigma_0}(0) 
\end{align*}
 which proves \eqref{induction1} for $k+1$. This also implies \eqref{induction2} for $k+1$ provided
 \begin{align*}
  ( k+1) C\sigma^\frac12 2^{5/2} A^3_{\sigma_0}(0) \le 1.
\end{align*}
 But the latter follows from \eqref{sigma} since 
 $$
 k+1\le n+1\le \frac T\delta+ 1 \le \frac {2T}\delta.
 $$
 
Finally,  the condition \eqref{sigma} is satisfied for $\sigma$ such that
    $$
    \frac{2T}{\delta} C\sigma^\frac12 2^{5/2} A^3_{\sigma_0}(0) =1.
    $$
Thus, 
$$
\sigma=c_1  T^{-2} , \ \ \text{where} \ c_1= \left( \frac{c_0}{C 2^{7/2} A^3_{\sigma_0}(0) (1+A_{\sigma_0}(0))^r }\right)^{1/2}
$$
which gives \eqref{siglb} if we choose $c\le c_1$.

\subsection{The general case: $s\in \R$}

For any $s\in \R$ we use the embedding \eqref{Gembedding} to get
\[
  u_0 \in G^{\sigma_0,s} \subset G^{\sigma_0/2}.
\]
From the local theory there is a $\delta=\delta \left(   A_{\sigma_0/2}(0) \right)$ such that
\[
  u \in C\left([0,\delta], G^{\sigma_0/2}\right).
\]
Fix an arbitrarily large $ T$. From the case $s=0$ in the previous subsection we have
\[
  u(t) \in G^{ 2 c_ \ast T^{-2} }  \quad \text{for $ t\in [0,   T]$},
\]
where $c_\ast > 0$ depends on $A_{\sigma_0/2}(0)$ and $\sigma_0$. Applying again the embedding \eqref{Gembedding} we conclude that
\[
  u(t) \in G^{  c_ \ast T ^{-2}, s }  \quad \text{for $ t\in [0,   T]$},
\]
 completing the proof of Theorem \ref{thm-gwp}.

\end{document}